\documentclass[11pt]{article}

\usepackage{amsmath,amsthm,amsfonts,amssymb,color}

\parskip=\medskipamount
\newcommand{\X}{\mathfrak{X}}
\newcommand{\Rc}{\mathfrak{R}}

\def\onehalf{{\textstyle\frac12}}

\def\lie#1{{\cal L}_{#1}}

\def\cycsum#1{\sum_{\stackrel{cyclic}{#1}}}
\def\cyclsum#1#2{\mathop{\sum_{\stackrel{cyclic}{#1}}}_{#2}}
\def\R{{\mathbb R}}

\def\hook{{\mathchoice{\vrule height 0pt depth 0.4pt width 3pt
\vrule height 5pt depth 0.4pt \kern 3pt} {\vrule height 0pt depth
0.4pt width 3pt \vrule height 5pt depth 0.4pt \kern 3pt} {\vrule
height 0pt depth 0.2pt width 1.5pt \vrule height 3pt depth 0.2pt
width 0.2pt \kern 1pt} {\vrule height 0pt depth 0.2pt width 1.5pt
\vrule height 3pt depth 0.2pt width 0.2pt \kern 1pt} }}

\theoremstyle{plain}
\newtheorem{thm}{Theorem}[section]
\newtheorem{lem}[thm]{Lemma}
\newtheorem{propn}[thm]{Proposition}
\newtheorem{cor}[thm]{Corollary}

\theoremstyle{definition}
\newtheorem{defn}[thm]{Definition}

\theoremstyle{remark}
\newtheorem*{rmk}{Remarks}

\newcommand{\pd}[2]{\frac{\partial#1}{\partial#2}}

\begin{document}

\title{ An intrinsic and exterior form of the Bianchi identities}


\author{Thoan Do and Geoff Prince
\thanks{Email: {\tt dtkthoan\char64 ctu.edu.vn,
geoff\char64 amsi.org.au}}
\\Department of Mathematics and Statistics, La Trobe University,\\ Victoria 3086,
Australia\\ The Australian Mathematical Sciences Institute,\\ c/o The University of Melbourne,\\Victoria 3010, Australia}

\maketitle

\abstract{We give an elegant formulation of the structure equations (of Cartan) and the Bianchi identities in terms of exterior calculus without reference to a particular basis and without the exterior covariant derivative. This approach allows both structure equations and the Bianchi identities to be expressed in terms of forms of arbitrary degree. We demonstrate the relationship with both the conventional vector version of the Bianchi identities and to the exterior covariant derivative approach. Contact manifolds, codimension one foliations and the Cartan form of classical mechanics are studied as examples of its flexibility and utility.}

\section{Introduction}\label{Intro}

The Bianchi identities are well-known in the literature and continue to be heavily used in mathematical physics.  They are realised in many ways, a number of which we find inelegant and obscure. It is our intention here to develop an index-free presentation of the structure equations and the Bianchi identities which relies solely on covariant and exterior derivatives and avoids matrix-valued forms and the exterior covariant derivative found in the classic texts of Kobayashi and Nomizu \cite{KN63} and Spivak \cite{Sp79}. Not only is the result more intrinsic and flexible than the conventional presentation, it is particularly amenable in situations where there is a distinguished set of one forms, for example, in the case of a contact manifold, a codimension one foliation or a distinguished co-frame. As a further extension we give a $p$-form version of the structure equations and the Bianchi identities which we apply in the context of the inverse problem in the calculus of variations.

An important aspect of our approach is that it allows the covariant derivative, torsion and curvature of the connection to be explicitly embedded in the expression of the exterior derivative of any $p$ form. We believe this to be a new and powerful result.

We will be dealing only with the Bianchi identities arising from a linear connection on (the tangent bundle of) a smooth manifold. We won't be dealing with connections on more general vector bundles. We will refer to the books of Crampin and Pirani \cite{cramp86} and Renteln~\cite{R14}.

For a given linear connection $\nabla$ on a smooth, connected manifold $M^n$, the torsion $T$ and the curvature $\Rc$ of the connection are defined as following:
\begin{align}
T(X,Y):=&\nabla_XY-\nabla_YX-[X,Y], \label{Torsion} \\
\Rc(X,Y)Z:=&\nabla_X\nabla_YZ-\nabla_Y\nabla_XZ-\nabla_{[X,Y]}Z, \label{Curvature}
\end{align}
for all $X,Y,Z \in \X(M).$

By taking covariant derivatives of $T$ and of $\Rc$, we obtain {\em the first Bianchi identity} and {\em the second Bianchi identity} respectively. These give a number of useful relations between the various operators and their derivatives, as given in Crampin and Pirani \cite{cramp86}. They are:
\begin{subequations}
\begin{align}
\cycsum{XYZ} \Rc(X,Y)Z&=\cycsum{XYZ} \nabla_X T(Y,Z)+\cycsum{XYZ}  T(T(X,Y),Z),\label{first-bianchi-identity}\\
\cycsum{XYZ} \nabla_X \Rc(Y,Z)&=\cycsum{XYZ} \Rc(X,T(Y,Z))\label{second-bianchi-identity}
\end{align}
\end{subequations}
where $\displaystyle{\cycsum{XYZ}}$ indicates a cyclic sum over the arguments $X,Y,Z.$\\

In this paper we will show that the definitions of curvature and torsion can be rephrased, for arbitrary 1-forms $\theta$ and vector fields $Z$, as
\begin{align*}
T_\theta &=d\theta-\nabla\theta\wedge I\\
\Rc_{\theta,Z} &=d\omega_{\theta,Z}-\nabla\theta\wedge\nabla Z,
\end{align*}
and that the Bianchi identities become
\begin{align*}
dT_\theta&=\Rc_\theta+\nabla\theta\wedge T,\\
d\Rc_{\theta,Z}&=\nabla\theta\wedge\Rc_Z+\Rc_\theta\wedge\nabla Z.
\end{align*}

The Bianchi identities are conventionally represented in terms of forms by taking exterior derivatives of {\em Cartan's first and second structure equations} which are respectively  (see \cite{cramp86}),
\begin{align}
& d\theta^a+\omega^a_b\wedge\theta^b=\Theta^a, \label{first-structure-equation}\\
& d\omega^a_b+\omega^a_c\wedge\omega^c_b=\Omega^a_b,\label{second-structure-equation}
\end{align}
where $\{\theta^a\}$ is a local basis of 1-forms for $\bigwedge^1(M)$ dual to a local basis of vector fields $\{U_a\}$ on $M$. The {\em connection forms} $\omega^a_b$ are given by
$$\omega^a_b(V):=\theta^a(\nabla_V U_b),$$
for an arbitrary vector field V. The {\em torsion 2-forms} $\Theta^a$ and the {\em curvature 2-forms} $\Omega^a_b$ are defined by
$$\Theta^a(X,Y):=\theta^a(T(X,Y)),$$
$$\Omega^a_b(X,Y):=\theta^a(\Rc(X,Y)U_b).$$

\begin{rmk}
\item[1.] We emphasise that the definitions \eqref{Torsion} and \eqref{Curvature} can be recovered from the structure equations \eqref{first-structure-equation} and \eqref{second-structure-equation} from which they are themselves derived.

\item[2.] Renteln~\cite{R14} reports the use of an abbreviated notation for \eqref{first-structure-equation}:
$$\Theta=d\theta + \omega\wedge\theta,$$
but wisely eschews its use. We will shortly present an accurate version.

\end{rmk}
Taking exterior derivatives of the first and second structure equations \eqref{first-structure-equation},\eqref{second-structure-equation} we have (again see \cite{cramp86}):
\begin{subequations} \label{Bianchi-id-form}
\begin{align}
d\Theta^a+\omega^a_b\wedge\Theta^b=\Omega^a_b\wedge\theta^b,\label{first-Bianchi-form}\\
d\Omega^a_b+\omega^a_c\wedge\Omega^c_b=\Omega^a_c\wedge\omega^c_b.\label{second-Bianchi-form}
\end{align}
\end{subequations}
These are the Cartan versions of the first and the second Bianchi identities respectively.

There is yet another way to think about the Bianchi identities using the exterior covariant derivative, $d^\nabla$. This is defined on a tensor-valued $k$-form $A$ acting on $M$ as follows (see for example \cite{KMS93,SLK14}):
$$d^\nabla A:=\nabla A \quad \text{if} \ k=0$$
where $\nabla A$ is the tensor-valued 1-form $\nabla A(X):=\nabla_XA$, and, for $1\le k \le n = \text{dim}(M)$,
\begin{align}
&d^\nabla A(X_0,\dots, X_k):=\sum^k_{i=0}(-1)^i\nabla_{X_i}(A(X_0,\dots,\bar X_i,\dots,X_k))\notag\\
&\quad\quad +\sum_{0\le i\le j \le k}(-1)^{i+j}A([X_i,X_j],X_0,\dots,\bar X_i,\dots,\bar X_j,\dots,X_k)\label{ext_cov_deriv}
\end{align}
where a bar over an argument indicates that it is missing.

 Now we consider the torsion $T$ as a vector-valued two-form  acting on $\mathfrak{X}(M)$ and $\Rc$ as an endomorphism-valued two-form or as a one-form taking values in the vector-valued two forms,  that is,
$\Rc: (X,Y)\mapsto \Rc (X,Y)$ or $\Rc: Z\mapsto \Rc_Z$ where $\Rc_Z(X,Y):=\Rc(X,Y)Z$. Then the structure equations are
\begin{subequations}
\begin{align}
d^\nabla I=T \label{Delanoe-SI}\\
d^\nabla\circ\nabla=\Rc, \label{Delanoe-SII}
\end{align}
\end{subequations}
\noindent where $\nabla$ is the {\em covariant differential} with $\nabla\theta(X):=\nabla_X\theta$ and $\nabla Z(X):=\nabla_X(Z).$

The Bianchi identities are (see \cite{Delanoe,SLK14})
\begin{subequations}
\begin{align}
d^\nabla T&= \Rc \wedge I \label{Delanoe-BI}\\
d^\nabla \Rc&=0. \label{Delanoe-BII}
\end{align}
\end{subequations}
Here $I$ is the identity endomorphism on $\mathfrak{X}(M)$, in other contexts also called the {\em canonical connection or soldering form}, so that $\Rc\wedge I(X,Y,Z)=\displaystyle\cycsum{XYZ}\Rc(X,Y)Z$. These versions of the Bianchi identities are derived by Delanoe~\cite{Delanoe} from the definitions \eqref{Torsion} and \eqref{Curvature} without use of the usual exterior derivative $d$ or the structure equations \eqref{Delanoe-SI},\eqref{Delanoe-SII}, although \eqref{Delanoe-SI} is known in the literature.

In the next sections, we will produce our new formulae for Bianchi identities and show the equivalence with those given in \eqref{first-bianchi-identity}, \eqref{second-bianchi-identity} and \eqref{Delanoe-BI}, \eqref{Delanoe-BII}. We will also demonstrate the utility of our approach by considering contact manifolds and codimension one foliations as examples.

\section{New versions of the structure equations and Bianchi identities}
We now formulate the structure equations and the Cartan version of the Bianchi identities in an intrinsic manner without reference to a particular basis using only $d$ and $\nabla$. To do this we effectively turn the torsion, as a vector-valued 2-form, and the curvature, as an endomorphism-valued 2-form, into conventional 2-forms. This allows the use of conventional exterior calculus.

\begin{defn}
For any $\theta \in \bigwedge^1(M)$ and $X,Y,Z \in \mathfrak{X}(M)$ the torsion and curvature 2-forms, $T_\theta$ and $\Rc_{\theta,Z}$, are defined as follows:
\begin{align}
T_\theta(X,Y):=\theta(T(X,Y))=T(X,Y)(\theta)\\
\Rc_{\theta,Z}(X,Y):=\theta(\Rc(X,Y)Z)=\Rc(X,Y)(\theta,Z).
\end{align}
\end{defn}
Note: $T_{\theta}$ and $\Rc_{\theta,Z}$ are function-linear in $\theta$ and $Z$. We also define the curvature 3-form $\Rc_\theta$ by
\begin{equation}
\Rc_\theta (X,Y,Z):=\cycsum{XYZ}\Rc_{\theta, Z}(X,Y).
\end{equation}
\begin{defn}
For $\theta \in \bigwedge^1(M)$ and $X,Y,Z \in \mathfrak{X}(M)$, the 2-forms $\Xi_\theta$, $\Psi_{\theta,Z}$ and the connection 1-forms $\omega_{\theta,Z}$ are defined as follows:
\begin{align}
&\Xi_\theta(X,Y):=\nabla_Y\theta(X)-\nabla_X\theta (Y)  &\text{or}& &\Xi_\theta&:=-\nabla\theta\wedge I,\\
&\Psi_{\theta,Z}(X,Y):=(\nabla_Y\theta)(\nabla_X Z)-(\nabla_X\theta)(\nabla_Y Z) &\text{or}& &\Psi_{\theta,Z}&:=-\nabla\theta\wedge\nabla Z,\\
&\omega_{\theta,Z}(X):=\theta(\nabla_X Z)  &\text{or}& &\omega_{\theta,Z}&:=\theta\circ\nabla Z.
\end{align}
\end{defn}

%

While the torsion and curvature two-forms and the connection one-forms are central to the development of the conventional structure equations and Bianchi identities, they can be generalised in a way that is useful in applications.

\begin{defn}\label{def-T-Theta-Xi-Theta}
For an arbitrary $p$-form $\Theta$ and $X_1,X_2,\dots, X_{p+1} \in \mathfrak{X}(M)$, the $(p+1)$-forms $\Xi_\Theta$, $T_\Theta$ are defined as follows:
\begin{align}
\label{def-Xi-Theta} \Xi_\Theta(X_1,\dots,X_{p+1})&:=\sum_{1\leq i\leq p+1} (-1)^{i}\nabla_{X_i}\Theta(X_1,\dots, \bar X_i, \dots, X_{p+1})\\
\label{def-T-Theta} T_{\Theta}(X_1,\dots, X_{p+1})&:=\sum_{1\leq i\leq j\leq p+1} (-1)^{i+j+1}\Theta(T(X_i,X_j), X_1,\dots, \bar X_i, \dots, \bar X_j, \dots, X_{p+1})
\end{align}
\end{defn}

\begin{thm}
For arbitrary $\Theta \in \bigwedge^p(M)$ we have
\begin{equation}\label{first-struc-eq-p-form}
T_\Theta=d\Theta+\Xi_\Theta.
\end{equation}
\end{thm}

\begin{proof}
Using the exterior derivative formula for $p$-forms (in which an over bar indicates a missing argument), we have
\begin{align*}
d\Theta(X_1,\dots, X_{p+1})=&\sum_{1\leq i\leq p+1} (-1)^{i+1} X_i(\Theta(X_1,\dots, \bar X_i, \dots, X_{p+1}))\\
                   +&\sum_{1\leq i\leq j\leq p+1} (-1)^{i+j}\Theta([X_i,X_j], X_1,\dots, \bar X_i, \dots, \bar X_j, \dots, X_{p+1})\\
                   =&\sum_{1\leq i\leq p+1} (-1)^{i+1} \nabla_{X_i}\Theta(X_1,\dots, \bar X_i, \dots, X_{p+1})\\
                   &+\sum_{1\leq i\leq j\leq p+1}(-1)^{i+j-1}\big(\Theta(\nabla_{X_i}X_j, X_1, \dots, \bar X_i, \dots, \bar X_j, \dots, X_{p+1})\\
                   &\quad -\Theta(\nabla_{X_j}X_i, X_1, \dots, \bar X_i, \dots, \bar X_j, \dots, X_{p+1})\big )\\
                   &+\sum_{1\leq i\leq j\leq p+1} (-1)^{i+j+1}\Theta(T(X_i,X_j), X_1,\dots, \bar X_i, \dots, \bar X_j, \dots, X_{p+1})\\
                   &+\sum_{1\leq i\leq j\leq p+1} (-1)^{i+j}\big(\Theta(\nabla_{X_i}X_j, X_1, \dots, \bar X_i, \dots, \bar X_j, \dots, X_{p+1})\\
                   &\quad -\Theta(\nabla_{X_j}X_i, X_1, \dots, \bar X_i, \dots, \bar X_j, \dots, X_{p+1})\big )
\end{align*}
Thus
\begin{align*}
d\Theta(X_1,\dots, X_{p+1})=&\sum_{1\leq i\leq p+1} (-1)^{i+1} \nabla_{X_i}\Theta(X_1,\dots, \bar X_i, \dots, X_{p+1})\\
&+\sum_{1\leq i\leq j\leq p+1} (-1)^{i+j+1}\Theta(T(X_i,X_j), X_1,\dots, \bar X_i, \dots, \bar X_j, \dots, X_{p+1})\\
=-&\Xi_\Theta(X_1,\dots, X_{p+1})+T_\Theta(X_1,\dots, X_{p+1})
\end{align*}
as required.
\end{proof}

\begin{cor}
For arbitrary $\theta \in \bigwedge^1(M)$ we have the
\begin{equation}\label{first-struc-eq-form}
\text{first structure equations:}\quad\quad T_\theta=d\theta+\Xi_\theta \quad \text{or}\quad T_\theta =d\theta-\nabla\theta\wedge I,
\end{equation}
where as defined in \eqref{def-Xi-Theta}, $\Xi_\theta(X,Y)=\nabla_Y \theta (X)-\nabla_X \theta(Y)=-\nabla\theta\wedge I (X,Y).$
\end{cor}

%
%
\begin{rmk} It will be useful in applications to rephrase the first structure equation \eqref{first-struc-eq-form} and its antecedent \eqref{first-struc-eq-p-form} as

$$d\theta= T_\theta + \nabla\theta\wedge I \ \quad\  \text{and}\ \quad \ d\Theta= T_\Theta - \Xi_\Theta,$$
giving the exterior derivative in terms of the torsion and the covariant derivative.
\end{rmk}

\begin{defn}\label{def-forms-second-structure-eq}
For $\Theta \in \bigwedge^p(M)$ and $X_1,X_2,\dots, X_{p+1}, Z$ $\in\mathfrak{X}(M)$, the $(p+1)$-forms $\Rc_{\Theta,Z}$, $\Psi_{\Theta,Z}$ and $T_{\Theta,Z}$ and $p$-forms $\omega_{\Theta,Z}$ are defined as follows:
\begin{align}
&\omega_{\Theta,Z}(X_1, \dots, X_p):=\sum_{1\leq i\leq p}(-1)^{i+1}\Theta(\nabla_{X_i}Z, X_1, \dots, \bar X_i,\dots, X_p)\\
&\Rc_{\Theta,Z}(X_1, \dots, X_{p+1}):=\sum_{1\leq i\leq j\leq p+1}(-1)^{i+j+1}\Theta(\Rc(X_i,X_j)Z, X_1,\dots, \bar X_i, \dots, \bar X_j, \dots X_{p+1})\\
\notag &\Psi_{\Theta,Z}(X_1, \dots, X_{p+1}):=\sum_{1\leq i\leq j\leq p+1}(-1)^{i+j}\big(\nabla_{X_i}\Theta(\nabla_{X_j}Z, X_1, \dots,\bar X_i, \dots, \bar X_j,\dots, X_{p+1})\\
 &\qquad \qquad \qquad \qquad  -\nabla_{X_j}\Theta(\nabla_{X_i}Z,X_1, \dots,\bar X_i, \dots, \bar X_j,\dots, X_{p+1})\big)\\
&T_{\Theta,Z}(X_1, \dots, X_{p+1}):=\cyclsum{X_iX_jX_k}{1\leq i\leq j\leq k\leq p+1}(-1)^{i+j+k}\Theta(T(X_i,X_j),\nabla_{X_k} Z, \dots)
\end{align}
We define $T_{\theta,Z}=0$ for 1-forms $\theta$.
\end{defn}
\begin{thm}
For arbitrary $\Theta \in \bigwedge^p(M)$ and  $Z \in \mathfrak{X}(M)$ we have
\begin{equation}\label{second-struc-eq-form-p-form}
d\omega_{\Theta,Z}=\Rc_{\Theta,Z}-\Psi_{\Theta,Z}+T_{\Theta,Z}.
\end{equation}

\end{thm}
\begin{proof} Using the exterior derivative formula we have
\begin{align*}
d\omega_{\Theta,Z}&(X_1,\dots, X_{p+1})=\sum_{i=1}^{p+1} (-1)^{i+1} X_i(\omega_{\Theta,Z}(X_1,\dots, \bar X_i, \dots, X_{p+1}))\\
                   +&\sum_{1\leq i\leq j\leq p+1} (-1)^{i+j}\omega_{\Theta,Z}([X_i,X_j], X_1,\dots, \bar X_i, \dots, \bar X_j, \dots, X_{p+1}).
\end{align*}
With
\begin{align*}
&\sum_{i=1}^{p+1} (-1)^{i+1} X_i(\omega_{\Theta,Z}(X_1,\dots, \bar X_i, \dots, X_{p+1}))\\
=&\sum_{i<j}(-1)^{i+j+1}\big[X_i(\Theta(\nabla_{X_j} Z, X_1,\dots, \bar X_i, \dots, \bar X_j,\dots, X_{p+1}))\\
&\qquad \qquad -X_j(\Theta(\nabla_{X_i} Z, X_1,\dots, \bar X_i, \dots, \bar X_j,\dots, X_{p+1}))\big]
\end{align*}
\begin{align*}
&\sum_{i=1}^{p+1} (-1)^{i+1} X_i(\omega_{\Theta,Z}(X_1,\dots, \bar X_i, \dots, X_{p+1}))\\
=&\sum_{i<j}(-1)^{i+j+1}\big[\nabla_{X_i}\Theta(\nabla_{X_j} Z, X_1,\dots, \bar X_i, \dots, \bar X_j,\dots, X_{p+1})\\
&\qquad \qquad -\nabla_{X_j}\Theta(\nabla_{X_i} Z, X_1,\dots, \bar X_i, \dots, \bar X_j,\dots, X_{p+1})\\
&\qquad \qquad +\Theta(\nabla_{X_i}\nabla_{X_j}Z,X_1,\dots, \bar X_i, \dots, \bar X_j,\dots, X_{p+1})\\
&\qquad \qquad -\Theta(\nabla_{X_j}\nabla_{X_i}Z,X_1,\dots, \bar X_i, \dots, \bar X_j,\dots, X_{p+1})\big]\\
&+\cyclsum{X_iX_jX_k}{i<j<k}(-1)^{i+j+k}\big[\Theta(\nabla_{X_j}Z,\nabla_{X_i}X_k,X_1,\dots, \bar X_i, \dots, \bar X_j,\dots,\bar X_k,\dots, X_{p+1})\\
&\qquad \qquad -\Theta(\nabla_{X_i}Z,\nabla_{X_j}X_k,X_1,\dots, \bar X_i, \dots, \bar X_j,\dots,\bar X_k,\dots, X_{p+1})\big],
\end{align*}
and,
\begin{align*}
&\sum_{1\leq i\leq j\leq p+1} (-1)^{i+j}\omega_{\Theta,Z}([X_i,X_j], X_1,\dots, \bar X_i, \dots, \bar X_j, \dots, X_{p+1})\\
=&\sum_{1\leq i\leq j\leq p+1} (-1)^{i+j}\Theta(\nabla_{[X_i,X_j]}Z, X_1,\dots, \bar X_i, \dots, \bar X_j, \dots, X_{p+1})\\
&+\cyclsum{X_iX_jX_k}{i<j<k}(-1)^{i+j+k+1}\Theta([X_i,X_j],\nabla_{X_k}Z, X_1,\dots, \bar X_i, \dots, \bar X_j, \dots,\bar X_k,\dots, X_{p+1})\\
=&\sum_{1\leq i\leq j\leq p+1} (-1)^{i+j}\Theta(\nabla_{[X_i,X_j]}Z, X_1,\dots, \bar X_i, \dots, \bar X_j, \dots, X_{p+1})\\
&+\cycsum{X_iX_jX_k}{i<j<k}(-1)^{i+j+k}\Theta(T(X_i,X_j),\nabla_{X_k}Z, X_1,\dots, \bar X_i, \dots, \bar X_j, \dots,\bar X_k,\dots, X_{p+1})\\
&+\cyclsum{X_iX_jX_k}{i<j<k}(-1)^{i+j+k}\Theta(\nabla_{X_k}Z,\nabla_{X_i}X_j-\nabla_{X_j}X_i, X_1,\dots, \bar X_i, \dots, \bar X_j, \dots,\bar X_k,\dots, X_{p+1}),
\end{align*}
we obtain
\begin{align*}
&d\omega_{\Theta,Z}(X_1,\dots, X_{p+1})\\
&\qquad \qquad =\sum_{1\leq i\leq j\leq p+1}(-1)^{i+j
+1}\big[\Theta(\Rc(X_i,X_j)Z, X_1,\dots, \bar X_i, \dots, \bar X_j, \dots X_{p+1})\\
&\qquad \qquad \qquad +\nabla_{X_i}\Theta(\nabla_{X_j}Z, X_1, \dots,\bar X_i, \dots, \bar X_j,\dots, X_{p+1})\\
 &\qquad \qquad \qquad \qquad  -\nabla_{X_j}\Theta(\nabla_{X_i}Z,X_1, \dots,\bar X_i, \dots, \bar X_j,\dots, X_{p+1})\big]\\
 &\qquad \qquad +\cyclsum{X_iX_jX_k}{i<j<k}(-1)^{i+j+k}\Theta(T(X_i,X_j),\nabla_{X_k}Z, X_1,\dots, \bar X_i, \dots, \bar X_j, \dots,\bar X_k,\dots, X_{p+1}).
\end{align*}
Therefore $d\omega_{\Theta,Z}=\Rc_{\Theta,Z}-\Psi_{\Theta,Z}+T_{\Theta,Z}$ as required.
\end{proof}
\begin{cor}
For arbitrary $\theta \in \bigwedge^1(M)$ and  $Z \in \mathfrak{X}(M)$ we have the second structure equations:
\begin{equation}\label{second-struc-eq-form}
 \Rc_{\theta,Z}=d\omega_{\theta,Z}+\Psi_{\theta,Z}\quad \text{or}\quad \Rc_{\theta,Z} =d\omega_{\theta,Z}-\nabla\theta\wedge \nabla Z,
\end{equation}
where $\Psi_{\theta,Z}(X,Y)=\nabla_Y \theta (\nabla_X Z)-\nabla_X \theta (\nabla_Y Z)=-\nabla\theta\wedge\nabla Z (X,Y).$\\
\end{cor}

%
%

\begin{cor}\label{Bianchi Corr}
By taking exterior derivatives of these structure equations we have:
\begin{subequations} \label{Bianchi}
\begin{align}
&\text{Bianchi I:}\quad \quad dT_\theta=d\Xi_\theta=\Rc_\theta+\nabla\theta\wedge T,\label{Bianchi-form-1}\hfill\\
&\text{Bianchi II:}\quad \quad d\Rc_{\theta,Z}=d\Psi_{\theta,Z}=\nabla\theta\wedge\Rc_Z+\Rc_\theta\wedge\nabla Z.\hfill \label{Bianchi-form-2}
\end{align}
\end{subequations}
\end{cor}

\begin{rmk}
Corollary \ref{Bianchi Corr} demonstrates that the Bianchi identities are the exterior differential consequences of the structure equations. Since the identity $d^2=0$ is implied by the Jacobi identity it is clear that the Bianchi identities are redundant in the presence of the structure equations and the Jacobi identities.\newline
We could have generalised the Bianchi identities in terms of an arbitrary $p$-form $\Theta$ instead of a 1-form $\theta$ but this is sufficient for our purposes.
\end{rmk}

The proof of the last equality in the second Bianchi identity is deferred until proposition~\ref{Bianchi Propn2}, but we can demonstrate that for the first Bianchi identity after this lemma.
\begin{lem}
\begin{align*}
\Rc_\theta(X,Y,Z):=&\cycsum{XYZ}\Rc_{\theta,Z}(X,Y)\\
=&-\cycsum{XYZ}(\nabla_X\nabla_Y\theta(Z)-\nabla_Y\nabla_X\theta(Z)-\nabla_{[X,Y]}\theta(Z)).
\end{align*}
\end{lem}

\begin{proof}
\begin{align*}
\Rc_{\theta,Z}(X,Y):=\ &\theta(\nabla_X\nabla_YZ-\nabla_Y\nabla_XZ-\nabla_{[X,Y]}Z)\\
                        =\ &(\nabla_X(\theta(\nabla_YZ))-\nabla_X\theta(\nabla_YZ)-\nabla_Y(\theta(\nabla_XZ))\\
                        &+\nabla_Y\theta(\nabla_XZ))-(\nabla_{[X,Y]}(\theta(Z))-\nabla_{[X,Y]}\theta(Z))\\
                        =\ &\nabla_X(\nabla_Y(\theta(Z)))-\nabla_X(\nabla_Y\theta(Z))-\nabla_X\theta(\nabla_YZ)\\
                        &-\nabla_Y(\nabla_X(\theta(Z)))+\nabla_Y(\nabla_X\theta(Z))+\nabla_Y\theta(\nabla_XZ)\\
                        &-\nabla_{[X,Y]}(\theta(Z))+\nabla_{[X,Y]}\theta(Z)\\
                        =\ &-\nabla_X\nabla_Y\theta(Z)-\nabla_Y\theta(\nabla_XZ)+\nabla_Y\nabla_X\theta(Z)\\
                        &+\nabla_X\theta(\nabla_YZ)-\nabla_X\theta(\nabla_YZ)+\nabla_Y\theta(\nabla_XZ)+\nabla_{[X,Y]}\theta(Z)\\
                        =\ &-(\nabla_X\nabla_Y\theta(Z)-\nabla_Y\nabla_X\theta(Z)-\nabla_{[X,Y]}\theta(Z)).
\end{align*}
\end{proof}
\begin{propn} \label{Bianchi Prop1} The first Bianchi identity \eqref{Bianchi-form-1} can be written as
\begin{equation}\label{dT-Rtheta}
dT_{\theta}(X,Y,Z)=\cycsum{XYZ}\Rc_{\theta,Z}(X,Y)+\cycsum{XYZ}\nabla_X\theta(T(Y,Z)),
\end{equation}
equivalently,
\begin{equation}\label{dT-Rtheta2}
dT_\theta=\Rc_\theta +\nabla\theta\wedge T.
\end{equation}
\end{propn}
\begin{proof} We begin with the right hand side of \eqref{Bianchi-form-1}:
\begin{align*}
d\Xi_{\theta}(X,Y,Z)&=\cycsum{XYZ}\nabla_X(\Xi_{\theta}(Y,Z))-\cycsum{XYZ}\Xi_{\theta}([X,Y],Z)\\
                    &=\cycsum{XYZ}\nabla_X(\nabla_Z\theta(Y)-\nabla_Y\theta(Z))-\cycsum{XYZ}(\nabla_Z\theta([X,Y])-\nabla_{[X,Y]}\theta(Z))\\
                    &=\cycsum{XYZ}(\nabla_X\nabla_Z\theta(Y)+\nabla_Z\theta(\nabla_XY)-\nabla_X\nabla_Y\theta(Z)-\nabla_Y\theta(\nabla_XZ))\\
                    &-\cycsum{XYZ}(\nabla_Z\theta([X,Y])-\nabla_{[X,Y]}\theta(Z))\\
                    &=-\cycsum{XYZ}(\nabla_X\nabla_Y\theta(Z)-\nabla_Y\nabla_X\theta(Z)-\nabla_{[X,Y]}\theta(Z))\\
                    &+\cycsum{XYZ}\nabla_Z\theta(\nabla_XY)-\cycsum{XYZ}\nabla_Z\theta(\nabla_YX)-\cycsum{XYZ}(\nabla_Z\theta([X,Y])\\
                    &=\cycsum{XYZ}\theta(\Rc(X,Y),Z)+\cycsum{XYZ}\nabla_Z\theta(T(X,Y))\ \text{from the lemma}\\
                    &=\cycsum{XYZ}\Rc_{\theta,Z}(X,Y)+\cycsum{XYZ}\nabla_X\theta(T(Y,Z)).
\end{align*}
Hence
\begin{equation*}
dT_{\theta}(X,Y,Z)=\cycsum{XYZ}\Rc_{\theta,Z}(X,Y)+\cycsum{XYZ}\nabla_X\theta(T(Y,Z))
\end{equation*}
or
\begin{equation*}
dT_\theta=\Rc_\theta +\nabla\theta\wedge T.
\end{equation*}
\end{proof}
\section{Relationship to other versions}
In this section we apply due diligence to demonstrate that we really do have the Bianchi identities. The proofs will emphasise the flexibility and utility of our intrinsic approach.

In the proofs of equivalence that follow we have used the definitions of $T, \Rc, T_\theta, \Rc_{\theta,Z}, \omega_{\theta,Z}, \Theta^a, \Omega^a_b, \omega^a_b$ where indicated. In other words, we are generally demonstrating the equivalence of different forms of the Bianchi identities in the presence of equivalent forms of the definitions of $T$ and $\Rc$.
\subsection*{Vector versions}

\begin{propn}
The form version of the first Bianchi identities \eqref{Bianchi-form-1} are equivalent to the vector field version of the first Bianchi identities \eqref{first-bianchi-identity}.
\end{propn}

\begin{proof}
We have:
\begin{align*}
dT_{\theta}(X,Y,Z)&=\cycsum{XYZ}\nabla_X(\theta(T(Y,Z)))-\cycsum{XYZ}\theta(T([X,Y],Z)) \quad \text{(definitions of $d$ and $T_\theta$)}\\
                      &=\cycsum{XYZ}\nabla_X\theta(T(Y,Z))+\cycsum{XYZ}\theta(\nabla_XT(Y,Z)+\cycsum{XYZ}\theta(T(\nabla_XY,Z))\\
                      &+\cycsum{XYZ}\theta(T(Y,\nabla_XZ))-\cycsum{XYZ}\theta(T([X,Y],Z)) \quad \text{(product rule)}.
\end{align*}
Hence
\begin{align*}
dT_{\theta}(X,Y,Z) &=\cycsum{XYZ}\nabla_X\theta(T(Y,Z))+\cycsum{XYZ}\theta(\nabla_XT(Y,Z))+\cycsum{XYZ}\theta(T(\nabla_XY,Z))\\
                      &-\cycsum{XYZ}\theta(T(\nabla_YX,Z))-\cycsum{XYZ}\theta(T([X,Y],Z))\quad (\text{since}\ \cycsum{XYZ}=\cycsum{ZXY})\\
                      &=\cycsum{XYZ}\nabla_X\theta(T(Y,Z))+\cycsum{XYZ}\theta(\nabla_XT(Y,Z)+\cycsum{XYZ}\theta(T(T(X,Y),Z))
\end{align*}
where, in the last line, we have used the definition of $T$.
Using this in the left hand side of equation \eqref{dT-Rtheta} and because $\theta$ is arbitrary we get:
\begin{align*}
&\cycsum{XYZ}\theta(\nabla_XT(Y,Z))+\cycsum{XYZ}\theta(T(T(X,Y),Z))=\cycsum{XYZ}\theta(R(X,Y),Z)\\
\Leftrightarrow & \cycsum{XYZ} R(X,Y)Z=\cycsum{XYZ}\nabla_XT(Y,Z)+\cycsum{XYZ} T(T(X,Y),Z).
\end{align*}
\end{proof}
 The next proposition, demonstrating the flexibility of this formulation,  gives what appears at first to be a strange result.
\begin{propn}
The vector field version of the first Bianchi identities can be recovered by the 2 structure equations without further differentiation.
\end{propn}

\begin{proof}
Replacing $\theta$ by $\omega_{\theta,Z}$ in the first structure equation gives:
$$d\omega_{\theta,Z}=T_{\omega_{\theta,Z}}-\Xi_{\omega_{\theta,Z}}.$$
Substituting this into \eqref{second-struc-eq-form}, we get
\begin{align}\label{structure-eq-form}
\Rc_{\theta,Z}=T_{\omega_{\theta,Z}}-\Xi_{\omega_{\theta,Z}}+\Psi_{\theta,Z}.
\end{align}
Evaluating \eqref{structure-eq-form} on $(X,Y)$, we have:
\begin{align}\label{structure-eq-form-act}
\Rc_{\theta,Z}(X,Y)=T_{\omega_{\theta,Z}}(X,Y)-\Xi_{\omega_{\theta,Z}}(X,Y)+\Psi_{\theta,Z}(X,Y).
\end{align}
where
\begin{align*}
T_{\omega_{\theta,Z}}(X,Y)&=\omega_{\theta,Z}(T(X,Y))\\
                          &=\theta(\nabla_{T(X,Y)}Z)\\
                          &=\theta(T(T(X,Y),Z)+\nabla_Z T(X,Y)+[T(X,Y),Z]+T(\nabla_Z X,Y)+T(X,\nabla_Z Y)),\\
\Psi_{\theta,Z}(X,Y)&=(\nabla_Y\theta)(\nabla_X Z)-(\nabla_X\theta)(\nabla_Y Z)\\
                    &=\nabla_Y(\theta(\nabla_X Z))-\theta(\nabla_Y \nabla_X Z)
                    -\nabla_X(\theta(\nabla_Y Z))+\theta(\nabla_X \nabla_Y Z),\\
\Xi_{\omega_{\theta,Z}}(X,Y)&=\nabla_Y\omega_{\theta,Z}(X)-\nabla_X\omega_{\theta,Z}(Y)\\
                            &=\nabla_Y(\omega_{\theta,Z}(X))-\omega(\nabla_Y X)-\nabla_X(\omega_{\theta,Z}(Y))+\omega_{\theta,Z}(\nabla_X Y)\\
                            &=\nabla_Y(\theta(\nabla_X Z))-\theta(\nabla_{\nabla_Y X} Z)-\nabla_X(\theta(\nabla_YZ))+\theta(\nabla_{\nabla_X Y} Z).
\end{align*}
Substituting these into \eqref{structure-eq-form-act} gives:
\begin{align*}
\theta(\Rc(X,Y)Z)&=\theta(T(T(X,Y),Z)+\nabla_Z T(X,Y)+[T(X,Y),Z]+T(\nabla_Z X,Y)\\
                 &+T(X,\nabla_Z Y)+\nabla_X \nabla_Y Z-\nabla_Y \nabla_X Z-\nabla_{\nabla_X Y} Z+\nabla_{\nabla_Y X} Z)\\
                 &=\theta(T(T(X,Y),Z)+\nabla_Z T(X,Y)-T(\nabla_X Y,Z)+T(\nabla_Y X,Z)-[[X,Y],Z]\\
                 &+T(\nabla_Z X,Y)+T(X,\nabla_Z Y)+\nabla_X \nabla_Y Z-\nabla_Y \nabla_X Z+\nabla_Z \nabla_Y X -\nabla_Z \nabla_X Y).
\end{align*}
It follows that
\begin{align}\label{Bianchi-vector-form}
\cycsum{XYZ}(\Rc(X,Y)Z=\cycsum{XYZ}(T(T(X,Y),Z)+\nabla_X T(Y,Z))
\end{align}
as required.

\end{proof}

\begin{propn}\label{Bianchi Propn2}
The form version of the second Bianchi identity \eqref{Bianchi-form-2} is equivalent to the vector field version \eqref{second-bianchi-identity}.
\end{propn}


\begin{proof}
Acting the second Bianchi identity $d\Rc_{\theta,W}=d\Psi_{\theta,W}$ on a triple of vector fields $(X,Y,Z)$ gives (not every line in the calculation is included)

\begin{align*}
d\Rc_{\theta,W}(X,Y,Z)&=\cycsum{XYZ} \nabla_X(\Rc_{\theta,W}(Y,Z))-\cycsum{XYZ}  \Rc_{\theta,W}([X,Y],Z)\\
                          &=\cycsum{XYZ}\nabla_X(\theta(\Rc(Y,Z)W))-\cycsum{XYZ}  \theta(\Rc([X,Y],Z)W)\quad \text{(definition of $d$ and $\Rc_{\theta,W}$)}\\
                          &=\theta(\cycsum{XYZ}\nabla_X\Rc(Y,Z)W+\cycsum{XYZ}\Rc(T(X,Y),Z)W)\\
                          &+\cycsum{XYZ} \nabla_X\theta(\Rc(Y,Z)W)+\theta(\cycsum{XYZ}\Rc(Y,Z)\nabla_XW)\quad \text{(product rule and defn. of $T$)}
\end{align*}

and

\begin{align*}
d\Psi_{\theta,W}(X,Y,Z)&=\cycsum{XYZ}\nabla_X(\Psi_{\theta,W}(Y,Z))-\cycsum{XYZ}\Psi_{\theta,W}([X,Y],Z)\\
                           &=\cycsum{XYZ}\nabla_X\theta(\Rc(Y,Z)W)+\theta(\cycsum{XYZ}\Rc(Y,Z)\nabla_XW)\quad \text{(defn. of $d$, $\Psi_{\Theta,W}$, $\Rc$, etc.)}
\end{align*}

\noindent (and so $d\Psi_{\theta,W}=\nabla\theta\wedge\Rc_W+\Rc_\theta\wedge\nabla W$ as stated in corollary~\ref{Bianchi Corr}).
Hence
\begin{align*}
0&=d(\Rc_{\theta,W}-\Psi_{\theta,W})(X,Y,Z)\\
 &=\theta(\cycsum{XYZ}\nabla_X\Rc(Y,Z)W+\cycsum{XYZ}\Rc(T(X,Y),Z)W).
\end{align*}
Since $\theta, W$ are arbitrary
$$\cycsum{XYZ}\nabla_X\Rc(Y,Z)=\cycsum{XYZ}\Rc(X,T(Y,Z))$$
as required.
\end{proof}
\subsection*{Cartan versions}
Now we show that the equivalence of the Cartan versions of the Bianchi identities with ours.

\begin{propn}\label{with-cartan-version}
The first Bianchi identity \eqref{Bianchi-form-1} is equivalent to the one given in \eqref{first-Bianchi-form}.
\end{propn}

\begin{proof}
\noindent Forward direction.
\smallskip

We will use the dual bases $\{\theta^a\}$ and $\{U_b\}$ of section \ref{Intro} along with the various constructs that appear in \eqref{first-Bianchi-form}. Replacing $\theta$ in equation \eqref{Bianchi-form-1} by $\theta^a$ and then acting a triple of vector fields $(X,Y,Z)$ on it, we have:

\begin{equation}
dT_{\theta^a}(X,Y,Z)=\cycsum{XYZ}\Rc_{\theta^a,Z}(X,Y)+\cycsum{XYZ}\nabla_X\theta^a(T(Y,Z))),\label{new-old-Bianchi1-1}
\end{equation}
but
\begin{align}
\ dT_{\theta^a}(X,Y,Z)&=d\Theta^a(X,Y,Z), \text{as $T_{\theta^a}=\Theta^a$}. \label{new-old-Bianchi1-2}\\
\notag \text{Now}\ \cycsum{XYZ}\Rc_{\theta^a,Z}(X,Y)&+\cycsum{XYZ}\nabla_X\theta^a(T(Y,Z))\\
\notag    &=\cycsum{XYZ}\theta^b(Z)\Omega^a_b(X,Y)\\
\notag    &+\cycsum{XYZ}(X(\theta^a(T(Y,Z)))-\theta^a(X(\theta^b(T(Y,Z))U_b)+\theta^b(T(Y,Z))\nabla_XU_b))\\
\notag    &=\Omega^a_b\wedge\theta^b(X,Y,Z)-\cycsum{XYZ}\Theta^b(Y,Z)\theta^a(\nabla_XU_b)\\
          &=\Omega^a_b\wedge\theta^b(X,Y,Z)-\omega^a_b\wedge\Theta^b(X,Y,Z).\label{new-old-Bianchi1-3}
\end{align}

Combining \eqref{new-old-Bianchi1-1}, \eqref{new-old-Bianchi1-2} and \eqref{new-old-Bianchi1-3} gives
$$d\Theta^a+\omega^a_b\wedge\Theta^b=\Omega^a_b\wedge\theta^b.$$.

We remark that we did not use either \eqref{first-structure-equation},\eqref{second-structure-equation} or \eqref{first-struc-eq-form},\eqref{second-struc-eq-form} in this part of the proof.

\smallskip
\noindent Reverse direction.
\smallskip

Let $\phi=\phi_a\theta^a$ be an arbitrary 1-form, then \eqref{first-Bianchi-form} gives
\begin{align}
&\phi_ad\Theta^a+\phi_a\omega^a_b\wedge\Theta^b=\phi_a\Omega^a_b\wedge\theta^b\notag \\
\implies\ & d(\phi_a\Theta^a)-d\phi_a\wedge\Theta^a+\phi_a\omega^a_b\wedge\Theta^b=\phi_a\Omega^a_b\wedge\theta^b\label{tag1}
\end{align}
The first term on the left is $dT_\phi$. Evaluating \eqref{tag1} on an arbitrary triple $(X,Y,Z)$ we see that
\begin{align*}
d\phi_a\wedge\Theta^a(X,Y,Z)&=\ \cycsum{XYZ}X(\phi_a)T_{\theta^a}(Y,Z),\\
\phi_a\omega^a_b\wedge\Theta^b&=\ \cycsum{XYZ}\phi_a\theta^a(\nabla_XU_b)T_{\theta^b}(Y,Z)\\
&=\ \cycsum{XYZ}\left(X(\phi_b)T_{\theta^b}(Y,Z)-\nabla_X\phi(T(Y,Z))\right)\\
\phi_a\Omega^a_b\wedge\theta^b(X,Y,Z)&=\ \cycsum{XYZ}\phi_a\theta^a(\Rc(X,Y)U_b)\theta^b(Z)\ =\ \cycsum{XYZ}\phi(\Rc(X,Y)Z)
\end{align*}
and so \eqref{tag1} becomes
\begin{align*}
& dT_\phi(X,Y,Z)-\cycsum{XYZ}\nabla_X\phi(T(Y,Z))=\cycsum{XYZ}\phi(\Rc(X,Y)Z)\\
\iff \ &dT_\phi=\Rc_\phi+\nabla\phi\wedge T
\end{align*}
as required.

\end{proof}

\begin{propn}
The second Bianchi identity \eqref{Bianchi-form-2} is equivalent to the one given in \eqref{second-Bianchi-form}.
\end{propn}

\begin{proof}
Forward direction.
\smallskip

Again replacing $\theta$ in equation \eqref{Bianchi-form-2} by $\theta^a$ and then acting a triple of vector fields $(X,Y,Z)$ on it, we have

\begin{equation}
d\Rc_{\theta^a,W}(X,Y,Z)=d\Psi_{\theta^a,W}(X,Y,Z),\label{new-old-Bianchi2-1}
\end{equation}

\begin{align}
\notag d\Rc_{\theta^a,W}(X,Y,Z)&= \cycsum{XYZ}\nabla_X(\Rc_{\theta^a,W}(Y,Z))- \cycsum{XYZ}\Rc_{\theta^a,W}([X,Y],Z)\\
\notag                         &= \cycsum{XYZ}\nabla_X(W^b\Omega^a_b(Y,Z))- \cycsum{XYZ}W^b\Omega^a_b([X,Y],Z)\\
\notag                         &=W^b\cycsum{XYZ}(\nabla_X(\Omega^a_b(Y,Z))- \Omega^a_b([X,Y],Z))+\cycsum{XYZ}X(W^b)\Omega^a_b(Y,Z)\\
                               &=W^bd\Omega^a_b(X,Y,Z)+\cycsum{XYZ}X(W^b)\Omega^a_b(Y,Z),\label{new-old-Bianchi2-2}
\end{align}
\begin{align}
\notag d\Psi_{\theta^a,W}(X,Y,Z)&=\cycsum{XYZ}\nabla_X\theta^a(\Rc(Y,Z)W)+\cycsum{XYZ}\theta^a(\Rc(Y,Z)\nabla_XW)\\
\notag                          &=\cycsum{XYZ}(X(W^b\Omega^a_b(Y,Z))-X(W^b)\theta^a(\Rc(Y,Z)U_b)-W^b\theta^a(\nabla_X(\Rc(Y,Z)U_b)))\\
\notag                          &+\cycsum{XYZ}(X(W^b)\theta^a(\Rc(Y,Z)U_b)+W^b\theta^a(\Rc(Y,Z)\theta^c(\nabla_XU_b)U_c))\\
                             &=\cycsum{XYZ}X(W^b)\Omega^a_b(Y,Z))-W^b\omega^a_c\wedge\Omega^c_b(X,Y,Z)+ W^b\omega^c_b\wedge\Omega^a_c(X,Y,Z). \label{new-old-Bianchi2-3}
\end{align}

Combining \eqref{new-old-Bianchi2-1}, \eqref{new-old-Bianchi2-2} and \eqref{new-old-Bianchi2-3} we have
$$d\Omega^a_b+\omega^a_c\wedge\Omega^c_b=\Omega^a_c\wedge\omega^c_b.$$
We remark that we did not use either \eqref{first-structure-equation}, \eqref{second-structure-equation} or \eqref{first-struc-eq-form}, \eqref{second-struc-eq-form} in this part of the proof.

\smallskip
The proof of reverse direction follows along the same lines as the second part of the proof of Proposition \ref{with-cartan-version}.
\smallskip

\end{proof}

\subsection*{$\mathbf{d^\nabla}$ versions}

Before we establish the equivalence of the exterior covariant derivative form of the Bianchi identities with our formulation we will examine the $d^\nabla$ form of the structure equations \eqref{Delanoe-SI}, \eqref{Delanoe-SII}. It is straight forward to establish \eqref{Delanoe-SI} using the definition, \eqref{ext_cov_deriv}, of $d^\nabla$. The second structure equation follows from observing that
\begin{align*}
d^\nabla(\nabla Z)(X,Y) &=\nabla_X(\nabla_YZ)-\nabla_Y(\nabla_XZ)-\nabla_{[X,Y]}Z\\
&=\Rc(X,Y)Z
\end{align*}
so that
$$d^\nabla(\nabla Z)=\Rc_Z \quad \iff \quad d^\nabla\circ\nabla=\Rc$$
where $\Rc(Z):=\Rc_Z$ as defined in section 1.\newline
The $d^\nabla$ Bianchi identities follow by taking $d^\nabla$ of the $d^\nabla$ structure equations.

Delanoe \cite{Delanoe} attempted to show that the Bianchi identities followed as a direct consequence of $d^2=0$ and to do so used the exterior covariant derivative $d^\nabla$ formulation \eqref{Delanoe-BI} and \eqref{Delanoe-BII}.
However, he produced these formulas for Bianchi identities from the vector field version instead of directly from structure equations, only indirectly using the exterior derivative. We will now demonstrate their direct derivation from $d^2=0$ before discussing applications of them in the next section.

\begin{propn}
For $\theta\in\bigwedge^1(M)$ and $X,Y,Z,W \in\mathfrak{X}(M)$ the following formulae for $d^\nabla$ hold
\begin{align*}
&\ \theta(d^\nabla T(X,Y,Z))=\Rc_\theta(X,Y,Z)+d(T_\theta-\Xi_\theta)(X,Y,Z)\\
&\ d^\nabla\Rc(X,Y,Z)(\theta,W)=d(\Rc_{\theta,W}-\Psi_{\theta,W})(X,Y,Z)
\end{align*}
and hence \eqref{Delanoe-BI},\eqref{Delanoe-BII} are equivalent to \eqref{Bianchi-form-1} and \eqref{Bianchi-form-2} and so follow by taking the exterior derivatives of the structure equations \eqref{first-struc-eq-form} and \eqref{second-struc-eq-form}.
\end{propn}

\begin{proof}
Using the definition \eqref{ext_cov_deriv} of $d^\nabla$ we have
\begin{equation*}
d^\nabla T(X,Y,Z):=\cycsum{XYZ}\nabla_X(T(Y,Z))-\cycsum{XYZ}T([X,Y],Z)
\end{equation*}
and so
\begin{align*}
dT_\theta(X,Y,Z):=&\cycsum{XYZ}X(T_\theta(Y,Z))-\cycsum{XYZ}T_\theta([X,Y],Z)\\
=&\cycsum{XYZ}\nabla_X\theta(T(Y,Z))+\cycsum{XYZ}\theta(\nabla_X(T(Y,Z))-\cycsum{XYZ}\theta(T([X,Y],Z))\\
=&\cycsum{XYZ}\nabla_X\theta(T(Y,Z)) +\theta(d^\nabla T(X,Y,Z)).
\end{align*}
The result for $d^\nabla T$ now follows immediately from the expression for $d\Xi_\theta$ in the proof of proposition \ref{Bianchi Prop1}.

The result for $d^\nabla \Rc$ follows initially from the observation that the last two terms in the following expression cancel because of the cyclic sum
\begin{align*}
d^\nabla\Rc(X,Y,Z)(\theta,W)=&\ \theta(d^\nabla\Rc(X,Y,Z)W)\\
=&\ \theta(\cycsum{XYZ}(\nabla_X\Rc)(Y,Z)W)+\theta(\cycsum{XYZ}\Rc(T(X,Y),Z)W)\\
&+\theta(\cycsum{XYZ}(\nabla_XY,Z)W)+\theta(\cycsum{XYZ}\Rc(\nabla_YX,Z)W).
\end{align*}
The rest of the demonstration follows from the expression for $d(\Rc_{\theta,W}-\Psi_{\theta,W})$ which can be found in the proof of proposition \ref{Bianchi Propn2}.
\end{proof}

\section{Applications}
We give three examples. We will demonstrate the utility of this new formulation of the Bianchi identities on two well-known scenarios involving non-integrable and integrable distributions. In both cases there is a distinguished one-form which will play the role of $\theta$. We also apply our $p$-form versions to the Cartan form of classical mechanics in the context of the inverse problem in the calculus of variations.
\subsection*{Contact manifolds}

 We follow \cite{CH85,T89,YK84}. A $(2n+1)$ -- dimensional contact manifold $M$ is equipped with a global, nonzero one-form $\alpha$ satisfying $\alpha\wedge(d\alpha)^n\neq 0$ where the exponent indicates the $n$-fold wedge product. In the light of this condition the contact form $\alpha$ is maximally non-integrable. Associated to $\alpha$ is the Reeb field, $V\in \mathfrak{X}(M)$, satisfying $V\hook d\alpha=0$ and $\alpha(V)=1.$ There is a standard construction of a Riemannian metric $g$ and a $(1,1)$ tensor field $\Phi$ on $M$ having the properties that
$$g(V,X)=\alpha(X),\quad 2g(X,\Phi(Y))=d\alpha(X,Y),\quad \Phi^2(X)=-X+\alpha(X)V.$$
Hence $g(V,V)=1$ and, if the Levi-Civita connection of $g$ is $\nabla$, then
$$\nabla_VV=0,\quad \nabla_V\alpha=0,\quad \nabla_V\Phi=0,$$
so that $V$ is a unit geodesic tangent field and $\alpha$ and $\Phi$ are parallel transported along the integral curves of $V$. There are examples of contact manifolds with linear connections with torsion  for which $V$ is autoparallel (e.g.,~\cite{MP94}) but we will stick with the torsion-free Levi-Civita connection of $g$ for simplicity.

Applying the first structure equation \eqref{first-struc-eq-form} to $\alpha$ we have
$$d\alpha=\nabla\alpha\wedge I$$
and so
$$0=V\hook d\alpha=\nabla_V\alpha-\nabla\alpha(V) \iff \nabla\alpha(V)=0=\alpha(\nabla V)=\omega_{\alpha,V}. $$
Applying the second structure equation \eqref{second-struc-eq-form} to $\alpha$ and $V$:
$$\Rc_{\alpha,V}=\nabla\alpha\wedge\nabla V=0,$$
after a little manipulation using the result of the first structure equation,
so that $\Rc(X,Y)V$ is orthogonal to $V$.

The Bianchi identities for a Levi-Civita connection are
$$\Rc_\theta=0\ \text{and}\ d\Rc_{\theta,Z}=\nabla\theta\wedge\Rc_Z$$
(here $\theta$ and $Z$ are arbitrary) of which only
$$V\hook d\Rc_{\alpha,V}=\nabla\alpha\wedge(V\hook\Rc_V)$$
is interesting in this context, remembering that $\Rc_V$ takes values orthogonal to $V$.

\subsection*{Frobenius integrable 1-forms}
Now we turn to the contrasting case of a manifold $M^n$ with a linear connection, not necessarily metric, and a global codimension one foliation.
That is, we suppose there exists a globally non-zero $\theta\in \bigwedge^1(M)$ which is Frobenius integrable, so that $d\theta\wedge\theta=0$. The Frobenius integrability of $\theta$ is equivalent to  the closure under the Lie bracket of  the $(n-1)-$~dimensional distribution $D \subset\mathfrak{X}(M)$ with $\theta(D)=0$. (We don't distinguish $D$ as a sub-bundle of $TM$ from the submodule of $\mathfrak{X}(M)$ which it generates.) Suppose also that $V$ is a non-zero vector field such that $\mathfrak{X}(M)=Sp\{V\}\bigoplus D$ and $\theta(V)=1$. For the moment we place no additional conditions on the relationship between $D$ and $\nabla$.\newline
We will now rephrase the Frobenius  condition in terms of $\nabla$ and $T_\theta$.
\begin{propn}\label{FITpropn}
Let $\theta$ be a global, non-zero one-form on $M$. Then
\begin{equation}\label{FITorsion}
d\theta\wedge\theta=0\ \iff \ T_\theta|_D=-\nabla\theta\wedge I.
\end{equation}
\end{propn}
\begin{proof}
Evaluating the first structure equation on $(X,Y)$ with $X,Y \in D$
\begin{align*}
T_\theta(X,Y)&=d\theta(X,Y) -\nabla\theta\wedge I(X,Y)\\
&=-\theta([X,Y])-\nabla\theta\wedge I(X,Y),
\end{align*}
so $T_\theta|_D=-\nabla\theta\wedge I$ if and only if $[X,Y]\in D.$
\end{proof}

Next we introduce the notion of invariance of $\theta$, equivalently $D$, under $\nabla$. As usual, $D$ is said to be flat with respect to $\nabla$ if $\Rc|_D=0.$ However, the presence of torsion is generally an obstruction to the  construction  of  coordinates on the leaves of $D$ in which the components of the connection are zero. Instead of pursuing notions of flatness we follow the book by Bejancu and Farran~\cite{BF06} and consider connections {\em adapted to foliations}. For the moment suppose that $D$ is a distribution, not necessarily  integrable, of dimension $n-p$ and that $D'$ is a complementary distribution of dimension $p$ so that $\mathfrak{X}(M)=D\bigoplus D'$.

\begin{defn}

\begin{itemize}
\item[]
\item[(a)] A linear connection $\nabla$ on $M$ is said to be {\em adapted} to a distribution $D$ if
$$\nabla_XU\in D, \quad \forall X\in\mathfrak{X}(M),\ U\in D.$$

\item[(b)] A linear connection $\nabla$ is said to be an {\em adapted linear connection} if it is adapted to both $D$ and $D'$.
\end{itemize}
\end{defn}

We will not address the existence of an adapted linear connection (with torsion) for a pair $D,\ D'$, suffice it to say that the connection of Massa and Pagani~\cite{JP02,MP94} is such an example for $p=1.$

Bejancu and Farran~\cite{BF06} give the following proposition,
\begin{propn}
Let $\nabla$ be a linear connection and $D$ a distribution on a manifold $M$. Then $D$ is parallel with respect to $\nabla$ if and only if $\nabla$ is an adapted connection to $D$.
\end{propn}
Here {\em parallel} means that $D_x$ is mapped to $D_y$ by parallel transport along any piecewise smooth path in $M$ between an arbitrary pair of points $x,y \in M$.

Now we return to our Frobenius integrable 1-form, $\theta$, its annihilator $D$ and complementary distribution $D':=Sp\{V\}.$ Suppose that $\nabla$ is a linear connection adapted to $D,D'$. By taking covariant derivatives of $\theta(D)=0$ and $\theta(V)=1$ and using the adapted connection property we find
\begin{align*}
\nabla_X\theta&=\lambda_X\theta,\quad \nabla_XV=-\lambda_XV, \quad \forall X\in D;\\
\nabla_V\theta&=\lambda_V\theta,\quad \nabla_VV=-\lambda_VV
\end{align*}
for  $\lambda_X:=\nabla_X\theta(V),\ \lambda_V:=\nabla_V\theta(V)$. (We could at least locally rescale $\theta$ and $V$ so that $\nabla_VV=0$ and $\theta(V)=1$ but this changes nothing.) Applying these relations to \eqref{FITorsion} we see the role of torsion in the integrability of $D$ once more:
\begin{propn}
In the presence of an adapted linear connection $\nabla$,
$$d\theta\wedge\theta=0\ \iff \ T_\theta|_D=0.$$
That is, $D$ is integrable if and only if, for all $X,Y \in D,$\ $T(X,Y)\in D.$
\end{propn}
Turning to the second structure equations~\eqref{second-struc-eq-form} we observe that $\omega_{\theta,X}=0$ for all $X\in D$ and hence
$$\Rc_{\theta,X}=0,\ \forall X\in D,$$
that is,\ $\Rc(W,Z)X \in D$ for all $X\in D$ and all $W,Z \in \mathfrak{X}(M).$ This is also immediately obvious from the vector field definition, \eqref{Curvature}, of $\Rc$ and is independent of the integrability of $\theta$.
Using the structure equations and the integrability of $\theta,$ the Bianchi identities~\eqref{Bianchi-form-1},\eqref{Bianchi-form-2} give
$$dT_\theta|_D=0 \quad \text{and}\quad d\Rc_{\theta,X}=0,\ \forall X\in D,$$
which are non-trivial only if $n\ge 4.$

\subsection*{The Cartan 1-form of classical mechanics}

The Euler-Lagrange equations of classical mechanics can be formulated in terms of the the Cartan 1-form, $\theta_L$, as follows (for more details see \cite{KP08} and references therein). Suppose that $M^n$ is the configuration space with co-ordinates $(x^a)$ and that $E:=\R\times TM$ with co-ordinates $(t,x^a,u^a)$ is the evolution space. The evolution space is naturally equipped with a vertical sub-bundle $V(E)$ spanned by the $V_a:=\pd{}{u^a}$ and a contact distribution spanned by the $\theta^a:=dx^a-u^adt$. The vertical endomorphism, $S=V_a\bigotimes \theta^a$ is an intrinsic vector-valued 1-form.

Systems of second order ordinary differential equations $\ddot x^a=f^a(t,x^b,\dot x^b)$ are represented by a semi-spray $\Gamma$ on $E$:
$$\Gamma:=\pd{}{t}+u^a\pd{}{x^a}+f^a\pd{}{u^a}.$$
The  Massa-Pagani connection induced by $\Gamma$ on $E$, $\nabla$, is uniquely defined by the properties $\nabla\Gamma=0$, $\nabla dt$, $\nabla S=0$ and that $V(E)$ is flat. The quantities $\Gamma^a_b:=-\onehalf\pd{f^a}{u^b}$ are important.

The interaction of the dynamics  with the geometric structure of $E$ is also manifested in the vector-valued 1-form $\lie{\Gamma}S$ which has eigenvalues $0,-1,1$ with corresponding eigenspaces spanned by $\Gamma, H_a,\text{and}\ V_a$ respectively where $H_a:=\pd{}{x^a}-\Gamma^b_a\pd{}{u^b}$. The dual eigenform basis is $\{dt,\theta^a,\psi^a\}$ with $\psi^a:=du^a-f^adt+\Gamma^a_b\theta^b$.

If a regular Lagrangian, $L(t,x^a,u^a)$, is specified then the Cartan 1-form on $E$ is
\begin{equation}
\theta_L:=Ldt+dL\circ S=Ldt+\pd{L}{u^a}(dx^a-u^adt)\label{Cartan 1-form}
\end{equation}
and the Euler-Lagrange equations can be expressed as
\begin{equation}
\Gamma_L \hook d\theta_L=0 \label{ELeqns}
\end{equation}
where
$$\Gamma_L:=\pd{}{t}+u^a\pd{}{x^a}+F^a\pd{}{u^a}$$
is the Euler-Lagrange field on $E$ and  $\ddot x^a=F^a$ are Newton's equations being the Euler-Lagrange equations in normal form. The necessary and sufficient conditions for a semispray to be an Euler-Lagrange field are known as the {\em Helmholtz conditions}, see \cite{KP08}.

Crampin, Prince and Thompson \cite{CPT84} showed that the Helmholtz conditions are equivalent to the existence of a closed, maximal rank $2$-form $\Omega:=g_{ab}\psi^a\wedge\theta^b$. Moreover, $g_{ab}=\displaystyle{\frac{\partial L^2}{\partial u^a\partial u^b}}$ and $\Omega = d\theta_L$ for each such multiplier $g_{ab}$.

Now using the Massa-Pagani connection the first structure equation \eqref{first-struc-eq-p-form}, we can replace the closure condition, $d\Omega=0$, by
$$T_\Omega=\Xi_\Omega.$$
Furthermore,  we can prove that
\begin{equation}d\Omega=0 \ \Leftrightarrow \ T_\Omega=\Xi_\Omega=0.\end{equation}
Hence the generalisation of first Bianchi identity is the sole necessary condition for the closure of $\Omega$, in this case, $dT_\Omega=\Rc_\Omega+\nabla \Omega \wedge T=0,$
where, for $\Theta \in \bigwedge^p(E),$
\begin{equation*}
\nabla\Theta\wedge T(X_1, \dots, X_{p+2}):=\cyclsum{X_iX_jX_k}{i<k<j}(-1)^{i+j+k+1}\nabla_{X_k}\Theta(T(X_i,X_j),X_1, \dots).
\end{equation*}

\noindent However, the version closest to the conventional Helmholtz conditions is
\begin{thm}
 If $\Omega \in \bigwedge^2(E)$ is of the form $$\Omega:=g_{ab}\psi^a\wedge\theta^b,$$
then $\Omega=d\theta_L$ for some regular Lagrangian $L$ if and only it has maximal rank and
\begin{align*}
& T_\Omega(\Gamma, V_a,V_b)=0,&&  T_\Omega(\Gamma, H_a,H_b)=0,\\
&\Xi_\Omega(\Gamma, V_a,H_b)=0, && \Xi_\Omega(V_a,V_b,H_c)=0.
\end{align*}
\end{thm}

\section*{Acknowledgements}
Thoan Do gratefully acknowledges receipt of a Vietnamese government MOET-VIED scholarship and scholarship support from La Trobe University, and the hospitality of the Australian Mathematical Sciences Institute.
We thank Willy Sarlet, Tom Mestdag and Mike Crampin for their interest and helpful discussions. Geoff Prince thanks the Department of Mathematics at Ghent University for its hospitality. We thank a referee for helpful comments which clarified the presentation.

\end{document}